\UseRawInputEncoding
\documentclass[11pt,twoside, leqno]{article}

\usepackage{mathrsfs}
\usepackage{amssymb}
\usepackage{amsmath}
\usepackage{amsthm}
\usepackage{amsfonts}

\usepackage{color}

\allowdisplaybreaks

\def\rr{{\mathbb R}}
\def\rn{{{\rr}^n}}
\def\zz{{\mathbb Z}}

\def\nn{{\mathbb N}}
\def\fz{\infty}

\def\ccc{{\mathbb C}}

\def\cm{{\mathcal M}}
\def\cs{{\mathcal S}}

\def\cl{{\mathcal L}}

\def\cp{{\mathcal P}}
\def\az{\alpha}

\def\supp{{\rm{\,supp\,}}}
\def\esup{\mathop\mathrm{\,ess\,sup\,}}
\def\einf{\mathop\mathrm{\,ess\,inf\,}}

\def\ls{\lesssim}
\def\lz{\lambda}

\def\sz{\sigma}

\def\hs{\hspace{0.3cm}}

\def\dint{\displaystyle\int}

\def\r{\right}
\def\lf{\left}

\def\bint{{\ifinner\rlap{\bf\kern.30em--}
\int\else\rlap{\bf\kern.35em--}\int\fi}\ignorespaces}

\def\sbint{{\ifinner\rlap{\bf\kern.32em--}
\hspace{0.078cm}\int\else\rlap{\bf\kern.45em--}\int\fi}\ignorespaces}

\def\dfrac{\displaystyle\frac}
\def\dsup{\displaystyle\sup}

\newtheorem{theorem}{Theorem}[section]
\newtheorem{lemma}[theorem]{Lemma}
\newtheorem{corollary}[theorem]{Corollary}

\theoremstyle{definition}
\newtheorem{remark}[theorem]{Remark}
\newtheorem{definition}[theorem]{Definition}
\numberwithin{equation}{section}

\numberwithin{equation}{section}

\numberwithin{equation}{section}

\begin{document}

\arraycolsep=1pt

\title{\Large\bf The dual spaces of variable anisotropic Hardy-Lorentz spaces and continuity of a class of linear operators \footnotetext{\hspace{-0.35cm} {\it 2010 Mathematics Subject Classification}.
{Primary 42B20; Secondary 42B30, 46E30.}
\endgraf{\it Key words and phrases.} Anisotropy, Hardy-Lorentz space, atom, Calder\'on-Zygmund operator, BMO space.
\endgraf
E-mail:
\texttt{wangxjmath@126.com}(Wenhua WANG),
\texttt{atwangmath@163.com}(Aiting WANG).
\endgraf This work is partially supported by the projects of university
level planning in Qinghai Minzu University
(Grant No. 2022GH25).
\endgraf $^\ast$\,Corresponding author
}}
\author{\bf Wenhua WANG$^1$, Aiting WANG$^*$$^2$}
\date{\small 1.
School of Mathematics and Statistics,
Wuhan University,
Wuhan 430072, Hubei, P. R. China, ORCID iD:https://orcid.org/0000-0001-5863-2834.\\
2. School of Mathematics and Statistics,
Qinghai Minzu University, 810000, Qinghai, P. R. China. ORCID iD:https://orcid.org/0000-0001-5338-9097. }
\maketitle

\vspace{-0.8cm}

\begin{center}
\begin{minipage}{13cm}\small
{\noindent{\bf Abstract} \
In this paper, the authors obtain
the continuity of a class of linear operators
on variable anisotropic Hardy-Lorentz spaces.
  In addition, the authors also obtain that the dual space of variable anisotropic Hardy-Lorentz spaces is the anisotropic BMO-type spaces with variable exponents.
  This result is still new even when the exponent function $p(\cdot)$ is $p$. }
\end{minipage}
\end{center}

\section{Introduction}
\hskip\parindent
As is known to all, Hardy space on the Euclidean space $\rn$ is  a good substitutes of Lebesgue space $L^p(\rn)$ when $p\in(0,1]$,
 and plays an important role in haronmic analysis and $\mathbf{PDE}$s; see, for examples, \cite{cw77,fs72,lx21,s93,s60,sa89,t17}. Moreover, when studying the boundedness of some
operators in the critical case, the weak Hardy space $w\mathcal{H}^{p}(\rn)$ naturally appears and it is a good substitute of $\mathcal{H}^p(\rn)$. $w\mathcal{H}^{p}(\rn)$ with $p\in(0,\,1)$ was first introduced by Fefferman and Soria \cite{fs86} to find out the biggest space from which the Riesz transform is bounded to the weak Lebesgue space $wL^{1}(\rn)$. In 2007, Abu-Shammala and Torchinsky  \cite{at07} introduced the Hardy-Lorentz spaces $\mathcal{H}^{p,\,r}(\rn)$ for the full range $p\in(0,\,1]$ and $r\in(0,\,\infty]$, and
obtained some real-variable characterizations of this space.
In 2016, Liu et al. \cite{lyy16}
introduced the anisotropic Hardy-Lorentz space $\mathcal{H}^{p,\, r}_A({\mathbb {R}}^n)$ associated with a general expansive dilation $A$,  including the classical isotropic Hardy-Lorentz space of Abu-Shammala and Torchinsky.

As a generalization, variable exponent function spaces have their applications in fluid dynamics \cite{am02}, image processing \cite{clr06}, $\mathbf{PDE}$s and variational calculus \cite{f07,t17}.
Let $p(\cdot):\rn\rightarrow(0,\,\infty)$ be a variable exponent function. Recently, Liu et al. \cite{lyy17}
introduced the variable  anisotropic Hardy-Lorentz space $\mathcal{H}^{p(\cdot),\, r}_A({\mathbb {R}}^n)$, via the radial grand maximal function, and then established its some real-variable characterizations, respectively, in terms of atom, the radial and the non-tangential maximal functions. For more information about variable function spaces, see \cite{cf13,cw14,dhh11,l21,ns12,s13,t19,wl12}.

 To complete the theory of the variable anisotropic Hardy-Lorentz space $\mathcal{H}^{p(\cdot),\,r}_A({\mathbb {R}}^n)$, in this article, we obtain
the boundedness of a class of  Calder\'on-Zygmund operators
from $\mathcal{H}^{p(\cdot),\,r}_{A}(\rn)$ to variable Lorentz space $L^{p(\cdot),\,r}(\rn)$ and from $\mathcal{H}^{p(\cdot),\,r}_{A}(\rn)$ to itself. In addition, we also obtain the dual space of $\mathcal{H}^{p(\cdot),\,r}_ A(\mathbb{R}^n)$ is the anisotropic
BMO-type space with variable exponents.

Precisely, this article is organized as follows.

In Section \ref{s2}, we recall some notations and definitions
concerning expansive dilations,  the variable Lorentz space $L^{p(\cdot),\,r}(\rn)$ and the variable anisotropic Hardy-Lorentz space $\mathcal{H}^{p(\cdot),\,r}_A({\mathbb {R}}^n)$, via the radial grand maximal function.

%When the exponent function $p(\cdot),\,r$ is reduced to the constant exponent $p$, i.e., $p(\cdot),\,r:=p\in(0,\,1]$,
%the molecular decomposition of $H^{p(\cdot),\,r}_{A}$ in Theorem \ref{t2.6} is reduced to the molecular decomposition of %$H^{p}_{A}(\rn)=H^{p,\,p}_{A}$ in \cite[Theorem 3.9]{lyy16} and also the molecular decomposition of
%$H^\vz_A({\mathbb {R}}^n)$ with $\vz(x,\,t):=t^p$ in \cite[Theorem 2.10]{lffy16}; see Remarks \ref{r2.7} below for more details.

Section \ref{s3} is devoted to establishing
the boundedness of anisotropic convolutional $\delta$-type  Calder\'on-Zygmund operators
from $\mathcal{H}^{p(\cdot),\,r}_{A}(\rn)$ to $L^{p(\cdot),\,r}(\rn)$ and from $\mathcal{H}^{p(\cdot),\,r}_{A}(\rn)$ to itself.

In Section \ref{s4},
we prove that the dual space of $\mathcal{H}^{p(\cdot),\,r}_ A(\mathbb{R}^n)$ is the anisotropic
BMO-type space with variable exponents (see Theorem \ref{t5.1}).
 For this purpose, we first introduce a new kind of anisotropic BMO-type spaces with variable exponents $\mathcal{BMO}_A ^{p(\cdot),\,q,\,s}(\mathbb{R}^n)$ in Definition \ref{d4.1}, which includes the space $\mathrm{BMO}(\mathbb{R}^n)$ of John and Nirenberg \cite{jn61}. It is worth pointing
out that  this result is also new, when $\mathcal{H}^{p(\cdot),\,r}_{A}(\rn)$ is reduced to
$\mathcal{H}^{p,\,r}_{A}(\rn)$.

Finally, we make some conventions on notation.
Let $\nn:=\{1,\, 2,\,\ldots\}$ and $\zz_+:=\{0\}\cup\nn$.
For any $\az:=(\az_1,\ldots,\az_n)\in\zz_+^n:=(\zz_+)^n$, let
$|\az|:=\az_1+\cdots+\az_n$ and
$\partial^\az:=
\lf(\frac{\partial}{\partial x_1}\r)^{\az_1}\cdots
(\frac{\partial}{\partial x_n})^{\az_n}.$
In this article, we denote by $C$ a \emph{positive
constant} which is independent of the main parameters, but it may
vary from line to line.  For any $q\in[1,\,\infty]$, we denote by $q^{'}$ its conjugate index.
For any $a\in\rr$, $\lfloor a\rfloor$ denotes the
\emph{maximal integer} not larger than $a$.
%We also use $C_{(\az,\,\bz,\,\ldots)}$ to denote a positive
%constant depending on the indicated parameters $\az,\,\bz,\,\ldots$.
The \emph{symbol} $D\ls F$ means that $D\le
CF$. If $D\ls F$ and $F\ls D$, we then write $D\sim F$.
If a set $E\subset \rn$, we denote by $\chi_E$ its \emph{characteristic
function}. If there are no special instructions, any space $\mathcal{X}(\rn)$ is denoted simply by $\mathcal{X}$.

%%%%%%%%%%%%%%%%%%%%%%%%%%%%%%%%%%%%%%%%%%%%%%%%%%%%%%%%%%%%%%%%%%

%%%%%%%%%%%%%%%%%%%%%% section 2 %%%%%%%%%%%%%%%%%%%%%%%%%%%%%%%%%%

%%%%%%%%%%%%%%%%%%%%%%%%%%%%%%%%%%%%%%%%%%%%%%%%%%%%%%%%%%%%%%%%%%%%%
\section{Preliminaries \label{s2}}
\hskip\parindent
Firstly,
we recall the definitions of {{anisotropic dilations}}
on $\rn$; see \cite[p.\,5]{b03}. A real $n\times n$ matrix $A$ is called an {\it
anisotropic dilation}, shortly a {\it dilation}, if
$\min_{\lz\in\sz(A)}|\lz|>1$, where $\sz(A)$ denotes the set of
all {\it eigenvalues} of $A$. Let $\lz_-$ and $\lz_+$ be two {\it positive numbers} such that
$$1<\lz_-<\min\{|\lz|:\ \lz\in\sz(A)\}\le\max\{|\lz|:\
\lz\in\sz(A)\}<\lz_+.$$

By \cite[Lemma 2.2]{b03}, we know that, for a given dilation $A$,
there exist a number $r\in(1,\,\fz)$ and a set $\Delta:=\{x\in\rn:\,|Px|<1\}$, where $P$ is some non-degenerate $n\times n$ matrix, such that $$\Delta\subset r\Delta\subset A\Delta,$$ and we can
assume that $|\Delta|=1$, where $|\Delta|$ denotes the
$n$-dimensional Lebesgue measure of the set $\Delta$. Let
$B_k:=A^k\Delta$ for $k\in \zz.$ Then $B_k$ is open, $$B_k\subset
rB_k\subset B_{k+1} \ \ \mathrm{and} \ \ |B_k|=b^k,$$ here and hereafter, $b:=|\det A|$.
An ellipsoid $x+B_k$ for some $x\in\rn$ and $k\in\zz$ is called a {\it dilated ball}.
Denote
\begin{eqnarray}\label{e2.1}
\mathfrak{B}:=\{x+B_k:\ x\in \rn,\,k\in\zz\}.
\end{eqnarray}
Throughout the whole paper, let $\sigma$ be the {\it smallest integer} such that $2B_0\subset A^\sigma B_0$
and, for any subset $E$ of $\rn$, let $E^\complement:=\rn\setminus E$. Then,
for all $k,\,j\in\zz$ with $k\le j$, it holds true that
\begin{eqnarray}
&&B_k+B_j\subset B_{j+\sz},\label{e2.3}\\
&&B_k+(B_{k+\sz})^\complement\subset
(B_k)^\complement,\label{e2.4}
\end{eqnarray}
where $E+F$ denotes the {\it algebraic sum} $\{x+y:\ x\in E,\,y\in F\}$
of  sets $E,\, F\subset \rn$.

 Recall a \textit{quasi-norm}, associated with
dilation $A$, is a Borel measurable mapping
$\rho:\rr^{n}\to [0,\infty)$, satisfying
\begin{enumerate}
\item[\rm{(i)}] $\rho(x)>0$ for all $x \in \rn\setminus\{ \vec 0_n\}$,
here and hereafter, $\vec 0_n$ denotes the origin of $\rn$;

\item[\rm{(ii)}] $\rho(Ax)= b\rho(x)$ for all $x\in \rr^{n}$, where, as above, $b:=|\det A|$;

\item[\rm{(iii)}] $ \rho(x+y)\le H\lf[\rho(x)+\rho(y)\r]$ for
all $x,\, y\in \rr^{n}$, where $H\in[1,\,\fz)$ is a constant independent of $x$ and $y$.
\end{enumerate}

By \cite[Lemma 2.4]{b03},
we know that all homogeneous quasi-norms associated with a given dilation
$A$ are equivalent. Therefore, for a
fixed dilation $A$, in what follows, for convenience, we
always use the {\it{step homogeneous quasi-norm}} $\rho$ defined by setting,  for all $x\in\rn$,
\begin{equation*}
\rho(x):=\sum_{k\in\zz}b^k\chi_{B_{k+1}\setminus B_k}(x)\ {\rm
if} \ x\ne \vec 0_n,\hs {\mathrm {or\ else}\hs } \rho(\vec 0_n):=0.
\end{equation*}
By \eqref{e2.3}, we know that, for all $x,\,y\in\rn$,
$$\rho(x+y)\le b^\sz[\rho(x)+\rho(y)];$$  Moreover, $(\rn,\, \rho,\, dx)$ is a space of
homogeneous type in the sense of Coifman and Weiss \cite{cw77},
where $dx$ denotes the {\it $n$-dimensional Lebesgue measure}.

A measurable function $p(\cdot): \rn\rightarrow(0,\,\infty)$ is called a {\it variable exponent}. For any variable exponent $p(\cdot)$, let
\begin{eqnarray}\label{e2.5}
&&p_- :=\einf_{x\in\rn} p(x)\quad \mathrm{and} \ \ p_+ :=\esup_{x\in\rn} p(x).
\end{eqnarray}
Denote by $\cp$ the set of all variable exponents $p(\cdot)$ satisfying $0<p_-\leq p_+<\infty$.

Let $f$ be a measurable function on $\rn$ and $p(\cdot)\in\cp$. Define
$$\|f\|_{L^{p(\cdot)}}:=\inf\lf\{\lambda \in(0,\,\infty):\varrho_{p(\cdot)}(f/\lambda)\leq 1\r\},$$
where $$\varrho_{p(\cdot)}(f):=\int_\rn |f(x)|^{p(x)}\, dx.$$
Moreover, the {\it variable Lebesgue space} $L^{p(\cdot)}$ is defined to be the set of all measurable functions $f$
satisfying that $\varrho_{p(\cdot)}(f)<\infty$, equipped with the quasi-norm $\|f\|_{L^{p(\cdot)}}$.

\begin{remark}\rm{ \cite{lyy17}}\label{r2.1}
Let $p(\cdot)\in\cp$.
\begin{enumerate}
\item[\rm{(i)}]  For any
$r\in(0,\,\infty)$ and $f\in L^{p(\cdot)}$,
$\lf\|{|f|^r}\r\|_{L^{p(\cdot)}}=\lf\|f\r\|^r_{L^{rp(\cdot)}}.$
Moreover, for any $\mu\in\ccc$ and $f,g\in L^{p(\cdot)}$,
$\lf\|\mu f\r\|_{L^{p(\cdot),}}=|\mu|\lf\|f\r\|_{L^{p(\cdot)}}$ and
$\lf\|f+g\r\|^{\underline{p}}_{L^{p(\cdot)}}\leq
\lf\|f\r\|^{\underline{p}}_{L^{p(\cdot)}}+
\lf\|g\r\|^{\underline{p}}_{L^{p(\cdot)}},$
where
\begin{align}\label{e2.5.1}
\underline{p}:=\min\{p_-,\,1\}
\end{align}
with $p_-$ as in \eqref{e2.5}.
\item[\rm{(ii)}] For any function $f\in L^{p(\cdot)}$
with $\lf\|f\r\|_{L^{p(\cdot)}}>0$,
$\varrho_{p(\cdot)}(f/{\|f\|_{L^{p(\cdot)}}})=1$
and, for $\lf\|f\r\|_{L^{p(\cdot)}}\leq1$, then
$\varrho_{p(\cdot)}(f)\leq\lf\|f\r\|_{L^{p(\cdot)}}$.
\end{enumerate}
\end{remark}

\begin{definition}
Let $p(\cdot)\in\cp$. The variable Lorentz space $L^{p(\cdot),\,r}$ is defined to be the set of all measurable functions $f$ such that
\begin{eqnarray*}
\|f\|_{L^{p(\cdot),\,r}}:=
\lf\{ \begin{array}{ll}
\lf[\dint_0^\infty \lambda^r \lf\|\chi_{\{x\in\rn:|f(x)|>\lambda\}}\r\|^r_{L^{p(\cdot)}}
\dfrac {d\lambda}{\lambda}\r]^{1/r},& \ r\in(0,\,\fz),\\
\dsup_{\lambda\in(0,\,\fz)}\lf[\lambda\lf\|\chi_{\{x\in\rn:
|f(x)|>\lambda\}}\r\|_{L^{p(\cdot)}}\r], & \ r=\fz
\end{array}\r.
\end{eqnarray*}
is finite.
\end{definition}

We say that $p(\cdot)\in \cp$ satisfy the {\it globally log-H\"{o}lder continuous condition}, denoted by
$p(\cdot)\in C^{\log}$, if there exist two positive constants $C_{\log}(p)$ and $C_{\infty}$, and $p_{\infty}\in \rr$
such that, for any $x, y\in\rn$,
\begin{align*}
|p(x)-p(y)|\leq \frac{C_{\log}(p)}{\log(e+1/\rho(x-y))}
\end{align*}
and
\begin{align*}
|p(x)-p_{\infty}|\leq \frac{C_{\infty}}{\log(e+\rho(x))}.
\end{align*}

%The following variable Lorentz space $L^{p(\cdot),\,r}$ is known as a special case of the variable Lorentz space
%$L^{p(\cdot),\,r,\,q(\cdot)}$ investigated by Kempka and Vyb\'{i}ral in \cite{kv14}.

A $C^\infty$ function $\varphi$ is said to belong to the Schwartz class $\cs$ if,
for every integer $\ell\in\zz_+$ and multi-index $\alpha$,
$\|\varphi\|_{\alpha,\ell}:=\dsup_{x\in\rn}[\rho(x)]^\ell|\partial^\az\varphi(x)|<\infty$.
The dual space of $\cs$, namely, the space of all tempered distributions on $\rn$ equipped with the weak-$\ast$
topology, is denoted by $\cs'$. For any $N\in\zz_+$, let
\begin{eqnarray*}
\cs_N:=\lf\{\varphi\in\cs:\ \|\varphi\|_{\alpha,\ell}\leq1,\ |\alpha|\leq N,\ \ \ell\leq N\r\}.
\end{eqnarray*}
In what follows, for $\varphi\in \cs$, $k\in\zz$ and $x\in\rn$, let $\varphi_k(x):= b^{-k}\varphi\lf(A^{-k}x\r)$.

\begin{definition}\label{d2.p}
Let $\varphi\in \cs$ and $f\in \cs'$. For any given $N\in \nn$, the{\it{ radial grand maximal function}} $M_{N}(f)$ of $f\in \cs'$ is defined by setting,
for any $x\in\rn$,
\begin{eqnarray*}
M_{N}(f)(x):=\sup_{\varphi\in \cs_N}\sup_{ k\in\zz}|f*\varphi_{k}(x)|.
\end{eqnarray*}
\end{definition}

\begin{definition}\rm{\cite{lyy17}}\label{d2.4}
Let $p(\cdot)\in C^{\log}$, $r\in(0,\,\infty)$, $A$ be a dilation and $N\in[\lfloor({1/\underline{p}}-1)\ln b/\ln\lambda_{-}\rfloor+2,\,\infty)$, where $\underline{p}$
is as in \eqref{e2.5.1}. The{\it{ variable anisotropic Hardy-Lorentz space}} $\mathcal{H}_{A}^{p(\cdot),\,r}$ is defined as
\begin{eqnarray*}
\mathcal{H}_{A}^{p(\cdot),\,r}:=\lf\{f\in \cs':M_{N}(f)\in L^{p(\cdot),\,r}\r\}
\end{eqnarray*}
and, for any $f\in \mathcal{H}_{A}^{p(\cdot),\,r}$, let $\|f\|_{\mathcal{H}_{A}^{p(\cdot),\,r}}:=\|M_{N}(f)\|_{L^{p(\cdot),\,r}}$.
\end{definition}
\begin{remark} Let $p(\cdot)\in C^{\log}$, $r\in(0,\,\infty)$.
\begin{enumerate}
\item[\rm{(i)}]
  When
 $p(\cdot):=p$, where $p\in(0,\,\infty)$, the space $\mathcal{H}^{p(\cdot),\,r}_{A}$ is reduced to the anisotropic Hardy-Lorentz space $\mathcal{H}^{p,\,r}_{A}$ studied in \cite{lyy16}.
%is reduced to \cite[Definition 3.7]{lyy16}.
\item[\rm{(ii)}]
When $A:=2{\rm I}_{n\times n}$ and $p(\cdot):=p$,
the space $\mathcal{H}^{p(\cdot),\,r}_{A}$ is reduced to the Hardy-Lorentz space $\mathcal{H}^{p,\,r}$ studied in \cite{at07}.
\end{enumerate}
\end{remark}

\begin{definition}\rm{\cite{lyy17}}\label{d3.1}
Let $p(\cdot)\in\cp$, $q\in(1,\,\infty]$ and
$$s\in[\lfloor(1/{p_-}-1) {\ln b/\ln \lambda_-}\rfloor,\,\infty)\cap\zz_+$$ with $p_-$ as in \eqref{e2.5}. An {\it anisotropic $(p(\cdot),\,q,\,s)$-atom} is a measurable function $a$ on $\rn$ satisfying
\begin{enumerate}
\item[\rm{(i)}] (support) $\supp a:=\overline{\{x\in\rn:a(x)\neq 0\}}\subset B$, where $B\in\mathfrak{B}$ and $\mathfrak{B}$ is as in \eqref{e2.1};
\item[\rm{(ii)}] (size) $\|a\|_{L^q}\le \frac{|B|^{1/q}}{\|\chi_B\|_{L^{p(\cdot)}}}$;
\item[\rm{(iii)}] (vanishing moment) $\int_\rn a(x)x^\alpha dx=0$ for any $\alpha\in \mathbb{Z}^n_+$ with $|\alpha|\leq s$.
\end{enumerate}
\end{definition}

\begin{definition}\rm{\cite{lyy17}}
Let $p(\cdot)\in C^{\log}$, $r\in(0,\,\infty)$, $q\in(1,\,\infty]$,
$s\in[\lfloor(1/{p_-}-1) {\ln b/\ln \lambda_-}\rfloor,\,\infty)\cap\zz_+$ with $p_-$ as in \eqref{e2.5}
and $A$ be a dilation. The {\it anisotropic variable atomic Hardy-Lorentz space}
$\mathcal{H}^{p(\cdot),\,q,\,s,\,r}_{A,\,\rm{atom}}$
is defined to be the set of all distributions $f\in \cs'$ satisfying that there exists a sequence of
$(p(\cdot),\,q,\,s)$-atoms, $\{a^k_i\}_{i\in\nn,k\in\zz}$, supported, respectively,
on $\{x^k_i+B_{\ell^k_i}\}_{i\in\nn,k\in\zz}\subset\mathfrak{B}$ and a positive constant $\widetilde{C}$ such that
$\sum_{i\in\nn} \chi_{x^k_i+A^{j_0}B_{\ell^k_i}}(x)\leq\widetilde{C}$ for any $x\in\rn$ and $k\in\zz$, with some
$j_0\in\zz\, \backslash\, \nn$, and
\begin{align*}%\label{e3.1}
f=\sum_{k\in\zz} \sum_{i\in\nn} \lz^k_ia^k_i \ \ \mathrm{in\ } \ \cs',
\end{align*}
where $\lambda^k_i\sim 2^k\|\chi_{x^k_i+B_{\ell^k_i}}\|_{L^{p(\cdot)}}$
for any $k\in\zz$ and $i\in\nn$ with the equivalent positive constants independent of $k$ and $i$.

Moreover, for any $f\in \mathcal{H}^{p(\cdot),\,q,\,s,\,r}_{A,\,\rm{atom}}$, define
$$\|f\|_{\mathcal{H}^{p(\cdot),\,q,\,s,\,r}_{A,\,\rm{atom}}}
:=\inf\lf[\sum_{k\in\zz} \lf\|\lf\{\sum_{i\in\nn} \lf[\frac{\lambda^k_i\chi_{x^k_i+B_{\ell^k_i}}}
{\lf\|\chi_{x^k_i+B_{\ell^k_i}}\r\|_{L^{p(\cdot)}}}\r]
^{\underline{p}}\r\}^{1/\underline{p}}\r\|^r
_{L^{p(\cdot)}}\r]^{1/r}.$$
\end{definition}
%%%%%%%%%%%%%%%%%%%%%%%%%%%%%%%%%%%%%%%%%%%%%%%%%%%%%%%%%%%%%%%%%%%%%%

\section{The continuity of Calder\'on-Zygmund operators \label{s3}}
\hskip\parindent
In this section, we get
the continuity of anisotropic convolutional $\delta$-type Calder\'on-Zygmund operators
from $\mathcal{H}^{p(\cdot),\,r}_{A}$ to $L^{p(\cdot),\,r}$ or from $\mathcal{H}^{p(\cdot),\,r}_{A}$ to itself.

Let $\delta\in (0,\,\frac{\ln \lambda_{+}}{\ln b})$. We call a linear operator $T$ is an anisotropic convolutional $\delta$-type Calder\'on-Zygmund operator, if $T$ is bounded on $L^2$ with kernel $\mathcal{K}\in \cs'$ coinciding with a locally integrable function on $\mathbb{R}^{n}\setminus\{ \vec 0_n\}$, and satisfying that there exists a positive constant $C$ such that, for any $x,\,y\in \mathbb{R}^{n}$ with $\rho(x)>b^{2\sigma}\rho(y)$,
$$|\mathcal{K}(x-y)-\mathcal{K}(x)|\leq C\frac{[\rho(y)]^{\delta}}{[\rho(x)]^{1+\delta}}.$$
For any $f\in L^2$, define $T(f)(x):=\mathrm{p.v.}\,\mathcal{K}\ast f(x)$.

\begin{theorem}\label{t5.3x}
Let $p(\cdot)\in C^{\log}$, $r\in(0,\,\infty)$ and $\delta\in (0,\,\frac{\ln \lambda_{+}}{\ln b})$. Assume that $T$ is an anisotropic convolutional $\delta$-type Calder\'on-Zygmund operator. If $p_{-}\in (\frac{1}{1+\delta},\,1)$ with $p_{-}$ as in \eqref{e2.5}, then there exists a positive constant $C$ such that, for any $\mathcal{H}^{p(\cdot),\,r}_{A}$,
\begin{itemize}
\item[{\rm (i)}]
$
\|T(f)\|_{L^{p(\cdot),\,r}}\leq C\|f\|_{\mathcal{H}^{p(\cdot),\,r}_{A}};
$
\item[{\rm (ii)}]
$
\|T(f)\|_{\mathcal{H}^{p(\cdot),\,r}_{A}}\leq C\|f\|_{\mathcal{H}^{p(\cdot),\,r}_{A}}.
$
\end{itemize}
\end{theorem}
\begin{remark}
%When $p(\cdot),\,r:=p\in(0,\,1]$,
When $p(\cdot):=p$, Theorem \ref{t5.3x} coincides with \cite[Theorem 6.16]{lyy16}.
%the above results are also new
%Even when $\delta\in(0,\,{\ln \lambda_-}/{\ln b}]$ and $p\in(1/(1+\delta),\,1]$, we have $s=0$, $\beta=1+{\ln \lambda_-}/{\ln b}$ and $N\geq 1$.
%In this case, by an argument similar to that used in the proof of Theorem 6.16(i) and Theorem 6.16(ii) in \cite{lyy16},
%we know that Theorem \ref{t4.1} and Theorem \ref{t4.2} also hold true for $\delta$-type Calder\'on-Zygmund $T$ with the assumption that $N\geq 1$ and $T^*1=0$. %respectively.
% Moreover, when it comes back to the isotropic setting, i.e., $A:=2\rm{I}_{n\times n}$, the above results are also new.
\end{remark}

%we need several technical lemmas as follows.
% we begin with
%the following notion of anisotropic $(p(\cdot),\,r,\,q,\,s)$-atoms introduced in \cite[Definition 4.1]{lyy17}.
%\begin{definition}\label{d3.1}
%Let $p(\cdot),\,r\in\cp$, $q\in(1,\,\infty]$ and
%$s\in[\lfloor(1/{p_-}-1) {\ln b/\ln \lambda_-}\rfloor,\,\infty)\cap\zz_+$ with $p_-$ as in \eqref{e2.5}. An {\it anisotropic $(p(\cdot),\,r,\,q,\,s)$-atom} is a %measurable function $a$ on $\rn$ satisfying
%\begin{enumerate}
%\item[\rm{(i)}]  $\supp a\subset B$, where $B\in\mathfrak{B}$ and $\mathfrak{B}$ is as in \eqref{e2.1};
%\item[\rm{(ii)}] $\|a\|_{L^q}\le \frac{|B|^{1/q}}{\|\chi_B\|_{L^{p(\cdot),\,r}}}$;
%\item[\rm{(iii)}] $\int_\rn a(x)x^\alpha dx=0$ for any $\alpha\in \mathbb{Z}^n_+$ with $|\alpha|\leq s$.
%\end{enumerate}
%\end{definition}

To prove Theorem \ref{t5.3x}, we need some technical lemmas.

\begin{lemma}\rm{\cite[Theorem 4.8]{lyy17}}\label{l3.1}
Let $p(\cdot)\in C^{\log}$, $r\in(0,\,\infty)$, $q\in(\max\{p_+,\,1\},\,\infty]$
with $p_+$ as in \eqref{e2.5} and
$s\in[\lfloor(1/{p_-}-1) {\ln b/\ln \lambda_-}\rfloor,\,\infty)\cap\zz_+$ with $p_-$ as in \eqref{e2.5}.
Then $$\mathcal{H}^{p(\cdot),\,r}_{A}=\mathcal{H}^{p(\cdot),\,q,\,s,\,r}_{A,\,\rm{atom}}$$ with equivalent quasi-norms.
\end{lemma}

By the proof of \cite[Theorem 4.8]{lyy17}, we obtain the following conclusion, which plays an important role in this section.
\begin{lemma}\label{l4.1xy}
 Let $p(\cdot)\in C^{\log}$, $r\in(0,\,\infty)$, $q \in (1,\,\infty)$ and $s\in[\lfloor(1/{p_-}-1) {\ln b/\ln \lambda_-}\rfloor,\,\infty)\cap\zz_+$ with $p_-$ as in \eqref{e2.5}. Then, for any $f\in \mathcal{H}^{p(\cdot),\,r}_{A}\cap L^q$, there exist $\{\lambda_i^k\}_{{i\in\nn,k\in\zz}}\subset\mathbb{C}$, dilated balls $\{x_i^k+B_{\ell_i^k}\}_{i\in\nn,k\in\zz}\subset\mathfrak{B}$ and
 $(p(\cdot),\,\infty,\,s)$-atoms $\{a_i^k\}_{{i\in\nn,k\in\zz}}$ such that
 $$f=\sum_{k\in\zz}\sum_{i\in\mathbb N}\lambda^k_ia^k_i \ {\rm in} \ L^q \ {\rm and} \ \mathcal{H}^{p(\cdot),\,r}_{A},$$
 where the series also converges almost everywhere.
\end{lemma}

\begin{proof}
Let $f\in \mathcal{H}^{p(\cdot),\,r}_{A}\cap L^q$. For any $k\in\zz$, by the proof of \cite[Theorem 4.8]{lyy17}, we know that there exist
$$\lf\{x^k_i\r\}_{i\in\nn}\subset\Omega_k:=\lf\{x\in\rn:\ M_{N} f(x)>2^k\r\},\ \ \lf\{\ell^k_i\r\}_{i\in\nn,k\in\zz}\subset\zz,$$
a sequence of $(p(\cdot),\,\fz,\,s)$-atoms,  $\{a^k_i\}_{k\in\zz,i\in\nn}$, supported on $\{x^k_i+B_{\ell^k_i+4\sigma}\}_{k\in\zz,i\in\nn}$, respectively, and $\{\lambda^k_i\}_{ k\in\zz,i\in\nn}\subset\mathbb{C}$, such that
\begin{equation}\label{e5.1}
f=\sum_{k\in\mathbb{Z}}\sum_{i\in\mathbb N}\lambda^k_ia^k_i=:\sum_{k\in\mathbb{Z}}\sum_{i\in\mathbb N}b^k_i\ \ \rm{in} \ \ \cs',
\end{equation}
and for any $k\in\zz$ and $i\in\nn$, $\supp b^k_i\subset x^k_i+B_{\ell^k_i+4\sigma} \subset \Omega_k$,
\begin{equation}\label{e5.23}
\lf\|b^k_i\r\|_{L^\fz}\ls 2^k\ \ {\rm and}\ \ \sharp\lf\{j\in\nn:\  (x^k_i+B_{\ell^k_i+4\sigma})\cap(x^k_j+B_{\ell^k_j+4\sigma})\neq\emptyset\r\}\le R,
\end{equation}
where $R$ is as in \cite[Lemma 4.7]{lyy17}. Moreover, by $f\in \mathcal{H}^{p(\cdot),\,r}_{A}\cap L^q$, we have, for almost every $x\in\Omega_k$, there exists a $k(x)\in\zz$ such that $2^{k(x)}<M_{N} f(x)\le 2^{k(x)+1}$. From this, $\supp b^k_i\subset\Omega_k$ and \eqref{e5.23}, we deduce that, for a.e. $x\in\rn$,
\begin{align}\label{e5.24}
\sum_{k\in\zz}\sum_{i\in\nn}\lf|b^k_i(x)\r|
\sim&\sum_{k\in\zz,k\in(-\fz,k(x)]}\sum_{i\in\nn}\lf|b^k_i(x)\r|
\ls\sum_{k\in\zz,k\in(-\fz,k(x)]}\sum_{i\in\nn}2^{k}
\chi_{x^k_i+B_{\ell^k_i+4\sigma}}(x)\\\nonumber
\sim&\sum_{k\in(-\fz,k(x)]\cap\zz}2^k\sim M_{N} f(x).\nonumber
\end{align}
Therefore, there exists a subsequence of the series $\{\sum_{|k|<K}\sum_{i\in\zz}b^k_i\}_{K\in\nn}$, denoted still by itself without loss of generality, which converges to some measurable function $F$ almost everywhere in $\rn$.

It follows from \eqref{e5.24} that, for any $K\in\nn$ and a.e. $x\in\rn$,
\begin{align*}
\lf|F(x)-\sum_{|k|<K}b^k_i(x)\r|
&\ls|F(x)|+\sum_{k\in\zz,k\in(-\fz,k(x)]}\sum_{i\in\nn}\lf|b^k_i(x)\r|\\
&\ls|F(x)|+ M_{N} f(x)\ls M_{N} f(x).
\end{align*}
From this, the fact that $M_{N} (f)\in L^q$, and the dominated convergence theorem, we conclude that $F=\sum_{k\in\zz}\sum_{i\in\nn}b^k_i$ in $L^q$. By this and \eqref{e5.24}, we know $f=F\in L^q$ and hence
$$f=\sum_{k\in\zz}\sum_{i\in\nn}b^k_i\ \ {\rm in}\ \ L^q\ {\rm and} \ \mathcal{H}^{p(\cdot),\,r}_{A},$$
 and also almost everywhere.
\end{proof}

%\begin{align*}
%fgghg
%\end{align*}
In what follows,
we also need the definition of {\it anisotropic Hardy-Littlewood maximal function} $\cm(f)$. For any $f\in L_{\mathrm{loc}}^1$ and $x\in \rn$,
\begin{align}\label{e2.9}
\cm(f)(x):=\sup_{x\in B\in\mathfrak{B}}\frac{1}{|B|}\int_{B}|f(z)|\,dz,
%=\sup_{k\in\zz}\sup_{y\in x+B_k}\frac{1}{|B_k|}\int_{y+B_k}|f(z)|\,dz
\end{align}
where $\mathfrak{B}$ is as in \eqref{e2.1}.

\begin{lemma}\rm{\cite[Lemma 4.3]{lyy17}}\label{l3.5}
Let $q\in(1,\,\infty)$. Assume that $p(\cdot)\in C^{\log}$ satisfies $1<p_-\leq p_+ <\infty$, where $p_-$ and $p_+$ are as in \eqref{e2.5}.
Then there exists a positive constant $C$ such that, for any sequence $\{f_k\}_{k\in\nn}$ of measurable functions,
$$\lf\|\lf\{\sum_{k\in\nn}\lf[\cm(f_k)\r]^q\r\}^{1/q}\r\|_{L^{p(\cdot)}}\leq C
\lf\|\lf(\sum_{k\in\nn}|f_k|^q\r)^{1/q}\r\|_{L^{p(\cdot)}},$$
where $\cm$ denotes the Hardy-Littlewood maximal operator as in \eqref{e2.9}.
\end{lemma}

\begin{lemma}\rm{\cite[Lemma 4.5]{lwyy19}}\label{l4.2xx}
 Let $p(\cdot)\in C^{\log}$, $r\in(0,\,\infty)$ and $q\in(\max\{p_+,\,1\},\,\infty)$. Then $\mathcal{H}^{p(\cdot),\,r}_{A}\cap L^q$ is dense in $\mathcal{H}^{p(\cdot),\,r}_{A}$.
\end{lemma}

The following Lemma show that variable anisotropic Hardy-Lorentz space $\mathcal{H}^{p(\cdot),\,r}_{A}$ is complete. Its proof is similar to \cite[Lemma 3.9]{w21}, we only need to make some minor changes. To limit the length of this article, we omit the concrete details.
\begin{lemma}\label{l4.2xxx}
 Let $p(\cdot)\in C^{\log}$, $r\in(0,\,\infty)$. Then $\mathcal{H}^{p(\cdot),\,r}_{A}$ is complete.
\end{lemma}

\begin{proof}[Proof of Theorem \ref{t5.3x}]

By the density, we only prove that (i) holds true for any $f\in \mathcal{H}^{p(\cdot),\,r}_{A}\cap L^q$  with $q\in(1,\,\infty)\cap(p_{+},\,\infty)$. For any $f\in \mathcal{H}^{p(\cdot),\,r}_{A}\cap L^q$, from Lemma \ref{l4.1xy}, we know that there exist numbers $\{\lambda_i^k\}_{i\in\mathbb{N},\,k\in\zz}\subset\mathbb{C}$ and a sequence of  $(p(\cdot),\,q,\,s)$-atom, $\{a_i^k\}_{i\in\mathbb{N},k\in\zz}$, supported, respectively, on $\{x_i^k+B_{\ell_i^k}\}_{i\in\nn,k\in\zz}\subset\mathfrak{B}$ such that
$$f=\sum_{k\in\zz}\sum_{i\in\nn} \lz_{i}^ka_i^k \ \ \mathrm{in} \ L^q$$
where
$\lambda^k_i\sim 2^k\|\chi_{x_i^k+B_{\ell_i^k}}\|_{L^{p(\cdot)}}$
for any $k\in\zz$ and $i\in\nn$,
$\sum_{i\in\nn} \chi_{x^k_i+A^{j_0}B_{\ell^k_i}}(x)\ls 1$ with some $j_0\in\zz\,\backslash\,\nn$
for any $x\in\rn$ and $k\in\zz$, and,
\begin{align}\label{e3.2}
\|f\|_{\mathcal{H}^{p(\cdot),\,q,\,s,\,r}_{A,\,\rm{atom}}}
\sim\lf[\sum_{k\in\zz}2^{kr}\lf\|\sum_{i\in\nn} \chi_{x_i^k+B_{\ell_i^k}}\r\|^r
_{L^{p(\cdot)}}\r]^{1/r}.
\end{align}
By the fact that  $T$ is bounded on $L^q$, we  have
$$T(f)=\sum_{k\in\zz}\sum_{i\in\nn} \lz_{i}^kT(a_i^k) \ \ \mathrm{in} \ L^q.$$
Set
$$f=\sum_{k=-\infty}^{k_0-1} \sum_{i\in\nn} \lz^k_ia^k_i+\sum_{k=k_0}^{\infty} \sum_{i\in\nn} \lz^k_ia^k_i
=:F_1+F_2 \ {\rm{in}} \ \ L^q \,.$$
 Then
\begin{align}\label{e4.3}
\lf\|\chi_{\{x\in\rn:\,T(f)(x)>2^{k_0}\}}\r\|_{L^{p(\cdot)}}
\ls&\lf\|\chi_{\{x\in\rn:\,T(F_1)(x)>2^{k_0-1}\}}\r\|_{L^{p(\cdot)}}
+\lf\|\chi_{\{x\in{E_{k_0}}:\,T(F_2)(x)>2^{k_0-1}\}}\r\|_{L^{p(\cdot)}} \\\nonumber
&+\lf\|\chi_{\{x\in(E_{k_0})^\complement:\,T(f_2)(x)>2^{k_0-1}\}}\r\|_{L^{p(\cdot)}} \nonumber\\
=&:{\rm{I_1+I_2+I_3}}\,, \nonumber
\end{align}
where $${E_{k_0}}:=\bigcup^\infty_{k={k_0}}\bigcup_{i\in\nn}\lf(x_i^k+A^\sigma B_{\ell_i^k}\r).$$ Therefore,
\begin{align}\label{e4.4}
{\rm{I_1}}&\ls \lf\|\chi_{\{x\in\rn:\,\sum_{k=-\infty}^{k_0-1} \sum_{i\in\nn} \lz^k_iT(a^k_i)(x)\chi_{x_i^k+A^\sigma B_{\ell_i^k}}(x)>2^{k_0-2}\}}\r\|_{L^{p(\cdot)}}\\
&\hs+\lf\|\chi_{\{x\in\rn:\,\sum_{k=-\infty}^{k_0-1} \sum_{i\in\nn} \lz^k_iT(a^k_i)(x)\chi_{(x_i^k+A^\sigma B_{\ell_i^k})^\complement}(x)>2^{k_0-2}\}}\r\|_{L^{p(\cdot)}}\nonumber \\
&=:{\rm{I_{1,1}+I_{1,2}}}\,. \nonumber
\end{align}
For the term ${\mathrm{I_{1,1}}}$, from the fact that $T$ is bounded on $L^q$,  Remark \ref{r2.1}, \eqref{e3.2} and a similar proof of  \cite[(4.7)]{lyy17}, we deduce  that
\begin{eqnarray}\label{e4.5}
\lf[\sum_{k\in\zz}2^{kr}\lf({\rm{I_{1,1}}}\r)^r\r]^{1/r}\ls
\lf[\sum_{k\in\zz}2^{kr}\lf\|\sum_{i\in\nn} \chi_{x_i^k+B_{\ell_i^k}}\r\|^r
_{L^{p(\cdot)}}\r]^{1/r}\sim
\|f\|_{\mathcal{H}^{p(\cdot),\,q,\,s,\,r}_{A,\,\rm{atom}}}.
\end{eqnarray}

For the term $\rm{I_{1,2}}$,  from the H\"{o}lder inequality and the size condition of $a_{i}^k(x)$, we conclude that, for any $x\in (x_i^k+A^{\sigma}B_{\ell_i^k})^{\complement}$,
\begin{align*}
\lf|Ta_{i}^k(x)\r|
&\leq \int_{x_i^k+B_{\ell_i^k+\sigma}}\lf|\mathcal{K}(x-y)-\mathcal{K}(x-x_i^k)\r||a_{i}^k(y)|\,dy\\\nonumber
&\lesssim\int_{x_i^k+B_{\ell_i^k+\sigma}}\frac{\rho(y-x_i^k)^{\delta}}
{\rho(x-x_i^k)^{1+\delta}}\lf|a_{i}^k(y)\r|\,dy
\lesssim \frac{\lf|x_{i}^ k+B_{\ell_i^k}\r|^{\delta}}{\rho(x-x_i^k)^{1+\delta}}\lf\|a_i^k\r\|_{L^q}
\lf|x_{i}^k+B_{\ell_i^k}\r|^{1/{q'}}\\ \nonumber
&\lesssim \frac{\lf|x_{i}^k+B_{\ell_i^k}\r|^{1+\delta}}{\rho(x-x_i^k)^{1+\delta}}
\frac{1}{\lf\|\chi_{x_i^k+B_{\ell_i^k}}\r\|_{L^{p(\cdot)}}}
\lesssim\lf[\cm(\chi_{x_i^k+B_{\ell_i^k}})(x)\r]^{1+\delta}
\frac{1}{\lf\|\chi_{x_i^k+B_{\ell_i^k}}\r\|_{L^{p(\cdot)}}}.\nonumber
\end{align*}
By this and a similar estimate of
 \cite[p.\,374]{lyy17}, we obtain
\begin{eqnarray*}
\lf[\sum_{k\in\zz}2^{kr}\lf({\rm{I_{1,2}}}\r)^r\r]^{1/r}\ls
\lf[\sum_{k\in\zz}2^{kr}\lf\|\sum_{i\in\nn} \chi_{x_i^k+B_{\ell_i^k}}\r\|^r
_{L^{p(\cdot)}}\r]^{1/r}\sim
\|f\|_{\mathcal{H}^{p(\cdot),\,q,\,s,\,r}_{A,\,\rm{atom}}}.
\end{eqnarray*}
Therefore, it follows from \eqref{e4.4} and \eqref{e4.5} that
\begin{eqnarray}\label{e4.14}
\lf[\sum_{k\in\zz}2^{kr}\lf({\rm{I_{1}}}\r)^r\r]^{1/r}\ls
%\lf[\sum_{k\in\zz}2^{kq}\lf\|\sum_{i\in\nn} \chi_{Q^k_i}\r\|^q
%_{L^{p(\cdot)}}\r]^{\frac 1{q}}\sim
\|f\|_{\mathcal{H}^{p(\cdot),\,q,\,s,\,r}_{A,\,\rm{atom}}}.
\end{eqnarray}

For ${\rm{I_2}}$ and ${\rm{I_{3}}}$, by  a proof similar to those of
\cite[(4.12) and (4.13)]{lyy17}, we obtain
\begin{eqnarray}\label{e4.15}
\lf[\sum_{k\in\zz}2^{kr}\lf({\rm{I_2}}\r)^r\r]^{1/r}\ls
%\lf[\sum_{k\in\zz}2^{kq}\lf\|\sum_{i\in\nn} \chi_{Q^k_i}\r\|^q
%_{L^{p(\cdot)}}\r]^{\frac 1{q}}\sim
\|f\|_{\mathcal{H}^{p(\cdot),\,q,\,s,\,r}_{A,\,\rm{atom}}}
\ {\rm{and}} \
\lf[\sum_{k\in\zz}2^{kr}\lf({\rm{I_3}}\r)^r\r]^{1/r}\ls
%\lf[\sum_{k\in\zz}2^{kq}\lf\|\sum_{i\in\nn} \chi_{Q^k_i}\r\|^q
%_{L^{p(\cdot)}}\r]^{\frac 1{q}}\sim
\|f\|_{\mathcal{H}^{p(\cdot),\,q,\,s,\,r}_{A,\,\rm{atom}}}.
\end{eqnarray}
%and
%\begin{eqnarray}\label{e4.16}
%\lf[\sum_{k_0\in\zz}2^{k_0q}\lf({\rm{J_3}}\r)^q\r]^{\frac 1{q}}\ls
%\lf[\sum_{k\in\zz}2^{kq}\lf\|\sum_{i\in\nn} \chi_{Q^k_i}\r\|^q
%_{L^{p(\cdot)}}\r]^{\frac 1{q}}\sim
%\|f\|_{H^{p(\cdot),\,r,\,s,\,q}_{A}}.
%\end{eqnarray}

Combining the estimate  of \eqref{e4.3}, \eqref{e4.14} and \eqref{e4.15}, we obtain
\begin{align*}
\|T(f)\|_{L^{p(\cdot),\,r}}
&\sim\lf[\sum_{k\in\zz}2^{kr}\lf\|\chi_{\{x\in\rn:\,|Tf(x)|>2^{k}\}}\r\|^r
_{L^{p(\cdot)}}\r]^{1/r}\\
&\ls\lf[\sum_{k\in\zz}2^{kr}\lf({\rm{I_1}}\r)^r\r]^{1/r}+
\lf[\sum_{k\in\zz}2^{kr}\lf({\rm{I_2}}\r)^r\r]^{1/r}+
\lf[\sum_{k\in\zz}2^{kr}\lf({\rm{I_3}}\r)^r\r]^{1/r}\\
&\ls\|f\|_{\mathcal{H}^{p(\cdot),\,q,\,s,\,r}_{A,\,\rm{atom}}}\sim\|f\|_{\mathcal{H}^{p(\cdot),\,r}_{A}},
\end{align*}
 which implies that $T(f)\in L^{p(\cdot),\,r}$. This finishes the proof of Theorem \ref{t5.3x}(i).

Now we show Theorem \ref{t5.3x}(ii). By Lemma \ref{l4.2xx}, we only need to prove that (ii) holds true for any $f\in \mathcal{H}^{p(\cdot),\,r}_{A}\cap L^q$  with $q\in(1,\,\infty)\cap(p_{+},\,\infty)$. Let $f\in \mathcal{H}^{p(\cdot),\,r}_{A}\cap L^q$. From Lemma \ref{l4.1xy}, we know that there exist numbers $\{\lambda_i^k\}_{i\in\mathbb{N},\,k\in\zz}\subset\mathbb{C}$ and a sequence of  $(p(\cdot),\,q,\,s)$-atom, $\{a_i^k\}_{i\in\mathbb{N},k\in\zz}$, supported, respectively, on $\{x_i^k+B_{\ell_i^k}\}_{i\in\nn,k\in\zz}\subset\mathfrak{B}$ such that
$$f=\sum_{k\in\zz}\sum_{i\in\nn} \lz_{i}^ka_i^k \ \ \mathrm{in} \ L^q$$
where
$\lambda^k_i\sim 2^k\|\chi_{x_i^k+B_{\ell_i^k}}\|_{L^{p(\cdot)}}$
for any $k\in\zz$ and $i\in\nn$,
$\sum_{i\in\nn} \chi_{x^k_i+A^{j_0}B_{\ell^k_i}}(x)\ls 1$ with some $j_0\in\zz\,\backslash\,\nn$
for any $x\in\rn$ and $k\in\zz$, and,
\begin{align}\label{e3.2c}
\|f\|_{\mathcal{H}^{p(\cdot),\,q,\,s,\,r}_{A,\,\rm{atom}}}
\sim\lf[\sum_{k\in\zz}2^{kr}\lf\|\sum_{i\in\nn} \chi_{x_i^k+B_{\ell_i^k}}\r\|^r
_{L^{p(\cdot)}}\r]^{1/r}.
\end{align}
By the fact that  $T$ is bounded on $L^q$, we  have
$$T(f)=\sum_{k\in\zz}\sum_{i\in\nn} \lz_{i}^kT(a_i^k) \ \ \mathrm{in} \ L^q.$$
Set
$$f=\sum_{k=-\infty}^{k_0-1} \sum_{i\in\nn} \lz^k_ia^k_i+\sum_{k=k_0}^{\infty} \sum_{i\in\nn} \lz^k_ia^k_i
=:F_1+F_2 \ {\rm{in}} \ \ L^q \,.$$
 Then
\begin{align}\label{e4.3c}
&\lf\|\chi_{\lf\{x\in\rn:\,M_N(T(f))(x)>2^{k_0}\r\}}\r\|_{L^{p(\cdot)}}\\ \nonumber
\ls&\lf\|\chi_{\lf\{x\in\rn:\,M_N(T(F_1))(x)>2^{k_0-1}\r\}}\r\|_{L^{p(\cdot)}}
+\lf\|\chi_{\lf\{x\in{G_{k_0}}:\,M_N(T(F_2))(x)>2^{k_0-1}\r\}}\r\|_{L^{p(\cdot)}} \\ \nonumber
&+\lf\|\chi_{\lf\{x\in(G_{k_0})^\complement:\,M_N(T(f_2))(x)>2^{k_0-1}\r\}}\r\|_{L^{p(\cdot)}} \nonumber\\
=&:{\rm{J_1+J_2+J_3}}\,, \nonumber
\end{align}
where $M_N$ is as in Definition \ref{d2.p} and $${G_{k_0}}:=\bigcup^\infty_{k={k_0}}\bigcup_{i\in\nn}\lf(x_i^k+A^\sigma B_{\ell_i^k}\r).$$ Therefore,
\begin{align}\label{e4.4c}
{\rm{J_1}}&\ls \lf\|\chi_{\{x\in\rn:\,\sum_{k=-\infty}^{k_0-1} \sum_{i\in\nn} \lz^k_iM_N(T(a^k_i))(x)\chi_{x_i^k+A^\sigma B_{\ell_i^k}}(x)>2^{k_0-2}\}}\r\|_{L^{p(\cdot)}}\\
&\hs+\lf\|\chi_{\{x\in\rn:\,\sum_{k=-\infty}^{k_0-1} \sum_{i\in\nn} \lz^k_iM_N(T(a^k_i))(x)\chi_{(x_i^k+A^\sigma B_{\ell_i^k})^\complement}(x)>2^{k_0-2}\}}\r\|_{L^{p(\cdot)}}\nonumber \\
&=:{\rm{J_{1,1}+J_{1,2}}}\,. \nonumber
\end{align}
For the term ${\mathrm{J_{1,1}}}$, from the fact that $M_N$ and $T$ are bounded on $L^q$,  Remark \ref{r2.1}, \eqref{e3.2c} and a similar proof of  \cite[(4.7)]{lyy17}, we deduce  that
\begin{eqnarray}\label{e4.5c}
\lf[\sum_{k\in\zz}2^{kr}\lf({\rm{J_{1,1}}}\r)^r\r]^{1/r}\ls
\lf[\sum_{k\in\zz}2^{kr}\lf\|\sum_{i\in\nn} \chi_{x_i^k+B_{\ell_i^k}}\r\|^r
_{L^{p(\cdot)}}\r]^{1/r}\sim
\|f\|_{\mathcal{H}^{p(\cdot),\,q,\,s,\,r}_{A,\,\rm{atom}}}.
\end{eqnarray}

For the term $\rm{J_{1,2}}$,  from the H\"{o}lder inequality, the size condition of $a_{i}^k(x)$, and a similar proof of \cite[p.117, Lemma]{s93}, we conclude that, for any $x\in (x_i^k+A^{\sigma}B_{\ell_i^k})^{\complement}$,
\begin{align*}
M_N(Ta_{i}^k)(x)
&=\sup_{\varphi\in\mathcal{S}_N}\sup_{j\in\zz}\lf|(\varphi_j*Ta_{i}^k)(x)\r|
=\sup_{\varphi\in\mathcal{S}_N}\sup_{j\in\zz}\lf|(\varphi_j*\mathcal{K}*a_{i}^k)(x)\r|\\
&\leq\sup_{\varphi\in\mathcal{S}_N}\sup_{j\in\zz} \int_{x_i^k+B_{\ell_i^k+\sigma}}
\lf|(\varphi_j*\mathcal{K})(x-y)-(\varphi_j*\mathcal{K})(x-x_i^k)\r|\lf|a_{i}^k(y)\r|\,dy\\\nonumber
&\lesssim\int_{x_i^k+B_{\ell_i^k+\sigma}}\frac{\rho(y-x_i^k)^{\delta}}
{\rho(x-x_i^k)^{1+\delta}}\lf|a_{i}^k(y)\r|\,dy
\lesssim \frac{\lf|x_{i}^ k+B_{\ell_i^k}\r|^{\delta}}{\rho(x-x_i^k)^{1+\delta}}\lf\|a_i^k\r\|_{L^q}
\lf|x_{i}^k+B_{\ell_i^k}\r|^{1/{q'}}\\\nonumber
&\lesssim \frac{\lf|x_{i}^k+B_{\ell_i^k}\r|^{1+\delta}}{\rho(x-x_i^k)^{1+\delta}}
\frac{1}{\lf\|\chi_{x_i^k+B_{\ell_i^k}}\r\|_{L^{p(\cdot)}}}
\lesssim\lf[\cm(\chi_{x_i^k+B_{\ell_i^k}})(x)\r]^{1+\delta}
\frac{1}{\lf\|\chi_{x_i^k+B_{\ell_i^k}}\r\|_{L^{p(\cdot)}}}.\nonumber
\end{align*}
From this and a similar estimate of
 \cite[p.\,374]{lyy17}, we deduce that
\begin{eqnarray*}
\lf[\sum_{k\in\zz}2^{kr}\lf({\rm{J_{1,2}}}\r)^r\r]^{1/r}\ls
\lf[\sum_{k\in\zz}2^{kr}\lf\|\sum_{i\in\nn} \chi_{x_i^k+B_{\ell_i^k}}\r\|^r
_{L^{p(\cdot)}}\r]^{1/r}\sim
\|f\|_{\mathcal{H}^{p(\cdot),\,q,\,s,\,r}_{A,\,\rm{atom}}}.
\end{eqnarray*}
Therefore, it follows from \eqref{e4.4c} and \eqref{e4.5c} that
\begin{eqnarray}\label{e4.14c}
\lf[\sum_{k\in\zz}2^{kr}\lf({\rm{J_{1}}}\r)^r\r]^{1/r}\ls
%\lf[\sum_{k\in\zz}2^{kq}\lf\|\sum_{i\in\nn} \chi_{Q^k_i}\r\|^q
%_{L^{p(\cdot)}}\r]^{\frac 1{q}}\sim
\|f\|_{\mathcal{H}^{p(\cdot),\,q,\,s,\,r}_{A,\,\rm{atom}}}.
\end{eqnarray}

For ${\rm{J_2}}$ and ${\rm{J_{3}}}$, by  a proof similar to those of
\cite[(4.12) and (4.13)]{lyy17}, we obtain
\begin{eqnarray}\label{e4.15c}
\lf[\sum_{k\in\zz}2^{kr}\lf({\rm{J_2}}\r)^r\r]^{1/r}\ls
%\lf[\sum_{k\in\zz}2^{kq}\lf\|\sum_{i\in\nn} \chi_{Q^k_i}\r\|^q
%_{L^{p(\cdot)}}\r]^{\frac 1{q}}\sim
\|f\|_{\mathcal{H}^{p(\cdot),\,q,\,s,\,r}_{A,\,\rm{atom}}}
\ {\rm{and}} \
\lf[\sum_{k\in\zz}2^{kr}\lf({\rm{J_3}}\r)^r\r]^{1/r}\ls
%\lf[\sum_{k\in\zz}2^{kq}\lf\|\sum_{i\in\nn} \chi_{Q^k_i}\r\|^q
%_{L^{p(\cdot)}}\r]^{\frac 1{q}}\sim
\|f\|_{\mathcal{H}^{p(\cdot),\,q,\,s,\,r}_{A,\,\rm{atom}}}.
\end{eqnarray}
%and
%\begin{eqnarray}\label{e4.16}
%\lf[\sum_{k_0\in\zz}2^{k_0q}\lf({\rm{J_3}}\r)^q\r]^{\frac 1{q}}\ls
%\lf[\sum_{k\in\zz}2^{kq}\lf\|\sum_{i\in\nn} \chi_{Q^k_i}\r\|^q
%_{L^{p(\cdot)}}\r]^{\frac 1{q}}\sim
%\|f\|_{H^{p(\cdot),\,r,\,s,\,q}_{A}}.
%\end{eqnarray}

It follows from the estimates  of \eqref{e4.3c}, \eqref{e4.14c} and \eqref{e4.15c} that
\begin{align*}
\|T(f)\|_{\mathcal{H}_A^{p(\cdot),\,r}}&=\lf\|M_N(T(f))\r\|_{L^{p(\cdot),\,r}}\\
&\sim\lf[\sum_{k\in\zz}2^{kr}\lf\|\chi_{\{x\in\rn:\,|M_N(Tf)(x)|>2^{k}\}}\r\|^r
_{L^{p(\cdot)}}\r]^{1/r}\\
&\ls\lf[\sum_{k\in\zz}2^{kr}\lf({\rm{J_1}}\r)^r\r]^{1/r}+
\lf[\sum_{k\in\zz}2^{kr}\lf({\rm{J_2}}\r)^r\r]^{1/r}+
\lf[\sum_{k\in\zz}2^{kr}\lf({\rm{J_3}}\r)^r\r]^{1/r}\\
&\ls\|f\|_{\mathcal{H}^{p(\cdot),\,q,\,s,\,r}_{A,\,\rm{atom}}}\sim\|f\|_{\mathcal{H}^{p(\cdot),\,r}_{A}},
\end{align*}
 which implies that $T(f)\in \mathcal{H}_A^{p(\cdot),\,r}$.
Therefore, we complete the proof of Theorem \ref{t5.3x}.
\end{proof}

\section{ The dual space of variable anisotropic Hardy-Lorentz space $\mathcal{H}_{A}^{p(\cdot),\,r}$}\label{s4}
\hskip\parindent
In this section, we establish the dual space of $\mathcal{H}_{A}^{p(\cdot),\,r}$. More precisely, we prove that the dual space of $\mathcal{H}_{A}^{p(\cdot),\,r}$ is the variable anisotropic BMO-type space $\mathcal{BMO}_A^ {p(\cdot),\,q,\,s}$.

Now, we define two new variable anisotropic BMO-type space as follows. In this article, for any $m\in\mathbb{Z}_+$, we use $P_m$ to denote the set of polynomials on $\rn$ with order
not more than $m$.
 For any $B\in\mathfrak{B}$ and any
locally integrable function $g$ on $\rn$, we use $P^m_B(g)$ to denote the minimizing polynomial of $g$ with degree not greater than $m$, which means that $P^m_
B(g)$ is the unique
polynomial $f\in P_m$ such that, for any $h\in P_m$,
$$\int_{B}h(x)(g(x)-f(x))\,dx=0.$$

\begin{definition}\label{d4.1}
 Let $A$ be a given dilation, $p(\cdot)\in\mathcal{P}$, $s$ be a nonnegative integer and $q\in[1,\,\infty)$. Then {\it the variable anisotropic BMO-type space } $\mathcal{BMO}_A^{p(\cdot),\,q,\,s}$ is defined to be the set of all $f\in L^q_{\mathrm{loc}}$ such that
$$ \|f\|_{\mathcal{BMO}_A^{p(\cdot),\,q,\,s}}:= \sup_{B\in\mathfrak{B}}\inf_{P\in P_s}\frac{|B|^{1-1/q}}{\|\chi_B\|_{L^{p(\cdot)}}}
 \lf[\int_B|f(x)-P(x)|^q\,dx\r]^{1/q}<\infty$$
 where $\mathfrak{B}$ is as in \eqref{e2.1}.
\end{definition}
\begin{definition}\label{d4.1x}
 Let $A$ be a given dilation, $p(\cdot)\in\mathcal{P}$, $s$ be a nonnegative integer and $q\in[1,\,\infty)$. Then {\it the variable anisotropic BMO-type space} $\widetilde{\mathcal{BMO}}_A^{p(\cdot),\,q,\,s}$ is defined to be the set of all $f\in L^q_{\mathrm{loc}}$ such that
$$ \|f\|_{\widetilde{\mathcal{BMO}}_A^{p(\cdot),\,q,\,s}}:= \sup_{B\in\mathfrak{B}}\frac{|B|^{1-1/q}}{\|\chi_B\|_{L^{p(\cdot)}}}
 \lf[\int_B|f(x)-P_B^s(f)(x)|^q\,dx\r]^{1/q}<\infty$$
 where $\mathfrak{B}$ is as in \eqref{e2.1}.
\end{definition}
\begin{lemma}\rm{\cite[(8.9)]{b03}}\label{l5.20x}
 Let $q\in[1,\,\infty]$, $A$ be a given dilation, $f\in L^q_{\mathrm{loc}}$ and $s$ be a nonnegative integer and $B\in\mathfrak{B}$. Then
 there exists a positive constant $C$, independent of $f$ and $B$, such that
 $$\sup_{x\in B}\lf|P_B^s(f)(x)\r|\leq C\frac{\int_B|f(x)|\,dx}{|B|}.$$
\end{lemma}
\begin{lemma}\label{l5.2x}
 Let $A$ be a given dilation, $p(\cdot)\in\mathcal{P}$, $s$ be a nonnegative integer and $q\in[1,\,\infty)$. Then
 $$\mathcal{BMO}_A^{p(\cdot),\,q,\,s}=
 \widetilde{\mathcal{BMO}}_A^{p(\cdot),\,q,\,s}$$
 with equivalent quasi-norms.
\end{lemma}
\begin{proof}
By the above definition, it is easy to see that $$\mathcal{BMO}_A^{p(\cdot),\,q,\,s}\supseteq
 \widetilde{\mathcal{BMO}}_A^{p(\cdot),\,q,\,s}.$$
 Conversely, from Lemma \ref{l5.20x} and the H\"{o}lder inequality, we conclude that, for any
 $B\in\mathfrak{B}$, $Q\in P_s$,
 \begin{align*}
 \lf[\frac{1}{|B|}\int_B|P_B^s(Q-f)(x)|^q\,dx\r]^{1/q}
 \lesssim&\frac{1}{|B|}
 \int_B|Q(x)-f(x)|\,dx\\
 \lesssim&
 \lf[\frac{1}{|B|}\int_B|Q(x)-f(x)|^q\,dx\r]^{1/q}.
 \end{align*}
 Therefore, by the Minkowski inequality, we obtain
 \begin{align*}
 &\frac{|B|}{\|\chi_B\|_{L^{p(\cdot)}}}\lf[\frac{1}{|B|}\int_B|P_B^s(f)(x)-f(x)|
 ^q\,dx\r]^{1/q}\\
 =&\frac{|B|}{\|\chi_B\|_{L^{p(\cdot)}}}
 \lf[\frac{1}{|B|}\int_B|P_B^s(Q-f)(x)+f(x)-Q(x)|^q\,dx\r]^{1/q}\\
 \lesssim&\frac{|B|}{\|\chi_B\|_{L^{p(\cdot)}}}
 \lf[\frac{1}{|B|}\int_B|Q(x)-f(x)|^q\,dx\r]^{1/q},
 \end{align*}
 which implies that
 $$\mathcal{BMO}_A^{p(\cdot),\,q,\,s}\subseteq
 \widetilde{\mathcal{BMO}}_A^{p(\cdot),\,q,\,s}.$$
 This completes the proof of Lemma \ref{l5.2x}.
\end{proof}

\begin{lemma}\label{l5.2}
Let $A$ be a given dilation, $p(\cdot)\in C^{\mathrm{log}}$, $r\in(0,\,1]$, $s$ be a nonnegative integer and $q\in[1,\,\infty)$. Then, for any
 continuous linear functional $\cl$ on $\mathcal{H}^{p(\cdot),\,r}_{A}=\mathcal{H}^{p(\cdot),\,q,\,s,\,r}_{A,\,\rm{atom}}$,
\begin{align*} \|\cl\|_{(\mathcal{H}^{p(\cdot),\,q,\,s,\,r}_{A,\,\rm{atom}})^{*}}:=&\sup\lf\{|\cl(f)|:
\|f\|_{\mathcal{H}^{p(\cdot),\,q,\,s,\,r}_{A,\,\rm{atom}}}\leq1\r\}\\=&\sup\{|\cl(a)|
: a \ \mathrm{is} \ (p(\cdot),\,q,\,s)\mathrm{-atom}\},
\end{align*}
where $(\mathcal{H}^{p(\cdot),\,q,\,s,\,r}_{A,\,\rm{atom}})^{*}$ denotes the dual space of $\mathcal{H}^{p(\cdot),\,q,\,s,\,r}_{A,\,\rm{atom}}$.
\end{lemma}
\begin{proof}
 Let  $a$ be a $(p(\cdot),\,q,\,s)$-atom. Then we have that $\|a\|_{\mathcal{H}^{p(\cdot),\,q,\,s,\,r}_{A,\,\rm{atom}}}\leq1$. Therefore,
 $$\sup\{|\cl(a)|
: a \ \mathrm{is} \ (p(\cdot),\,q,\,s)\mathrm{-atom}\}\leq\sup\lf\{|\cl(f)|:
\|f\|_{\mathcal{H}^{p(\cdot),\,q,\,s,\,r}_{A}}\leq1\r\}.$$
Moreover, let $f\in \mathcal{H}^{p(\cdot),\,q,\,s,\,r}_{A,\,\rm{atom}}$ and $\|f\|_{\mathcal{H}^{p(\cdot),\,q,\,s,\,r}_{A,\,\rm{atom}}}\leq1$. Then, for any $\varepsilon>0$, we know that there exist $\{\lambda_i^k\}_{i\in\nn,k\in\zz}\subset\mathbb{C}$ and a sequence of
$(p(\cdot),\,q,\,s)$-atoms, ${\{a_i^k\}}_{i\in\nn,k\in\zz}$, supported, respectively,
on ${\{x_i^k+B_{\ell_i^k}}\}_{i\in\nn,k\in\zz}\subset\mathfrak{B}$ such that
\begin{align*}
f=\sum_{k\in\zz}\sum_{i\in\nn} \lz_{i}^ka_i^k \ \ \mathrm{in\ } \ \cs' \,\,\mathrm{and}\,\,\mathrm{a.}\,\, \mathrm{e.}
\end{align*}
and
\begin{eqnarray*}
 \lf[\sum_{k\in\zz}2^{kr}\lf\|\sum_{i\in\nn} \chi_{x_i^k+B_{\ell_i^k}}\r\|^r
_{L^{p(\cdot)}}\r]^{1/r}\leq1+\varepsilon.
\end{eqnarray*}
Therefore, from the boundedness of $\cl$, $\lambda^k_i\sim 2^k\|\chi_{x_i^k+B_{\ell_i^k}}\|_{L^{p(\cdot)}}$ and $r\in(0,\,1]$, we further conclude that
\begin{align*}
 |\cl(g)|\leq&\sum_{k\in\zz}\sum_{i\in\nn}\lf|\lambda_i^k\r||\cl(a_i^k)|
 \leq\sum_{k\in\zz}\sum_{i\in\nn}\lf|\lambda_i^k\r|\sup\{|\cl(a)|: a \ \mathrm{is} \ (p(\cdot),\,q,\,s)\mathrm{-atom}\}\\
 \lesssim&\lf[\sum_{k\in\zz}2^{kr}\lf\|\sum_{i\in\nn} \chi_{x_i^k+B_{\ell_i^k}}\r\|^r
_{L^{p(\cdot)}}\r]^{1/r}\sup\{|\cl(a)|: a \ \mathrm{is} \ (p(\cdot),\,q,\,s)\mathrm{-atom}\}\\
 \lesssim&(1+\varepsilon)\sup\{|\cl(a)|: a \ \mathrm{is} \ (p(\cdot),\,q,\,s)\mathrm{-atom}\}.
\end{align*}
Combined with the arbitrariness of $\varepsilon$ and hence finishes the proof of Lemma \ref{l5.2}.
\end{proof}

For any $q\in[1,\,\infty]$ and $s\in\mathbb{Z}_+$. Denote by $L_{\mathrm{comp}}^{q}$ the set of all functions $f\in L^q$ with
compact support and
$$L_{\mathrm{comp}}^{q,\,s}:=\lf\{f\in L_{\mathrm{comp}}^{q}: \int_{\rn}f(x)x^{\alpha}\,dx=0,\,|\alpha|\leq s\r\}.$$

The main result of this section is as follows.
\begin{theorem}\label{t5.1}
 Let $A$ be a given dilation, $r\in(0,\,1]$, $p(\cdot)\in C^{\mathrm{log}}$, $p_+\in(0,\,1]$, $q\in(\max\{p_+,\,1\},\,\infty)$ and $s\in[\lfloor(1/{p_-}-1) {\ln b/\ln \lambda_-}\rfloor,\,\infty)\cap\zz_+$ with $p_-$ as in \eqref{e2.5}. Then  $$(\mathcal{H}^{p(\cdot),\,r}_{A})^{*}= (\mathcal{H}^{p(\cdot),\,q,\,s,\,r}_{A,\,\rm{atom}})^{*}
= \mathcal{BMO}_A^ {p(\cdot),\,q',\,s}=\widetilde{\mathcal{BMO}}_A^{p(\cdot),\,q',\,s}$$
 in the following sense: for any
 $\psi\in\mathcal{BMO}_A^ {p(\cdot),\,q',\,s}$ or $ \widetilde{\mathcal{BMO}}_A^{p(\cdot),\,q',\,s}$, the linear functional
 \begin{align}\label{e5.20}
 \cl_{\psi}(g):=\int_{\rn}\psi(x)g(x)\,dx,
 \end{align}
 initial defined for all $g\in L^{q,\,s}_{\mathrm{comp}}$, has a bounded extension to $\mathcal{H}^{p(\cdot),\,q,\,s,\,r}_{A,\,\rm{atom}}=\mathcal{H}^{p(\cdot),\,r}_{A}$.

 Conversely, if $\cl$ is a bounded linear functional on $\mathcal{H}^{p(\cdot),\,q,\,s,\,r}_{A,\,\rm{atom}}=\mathcal{H}^{p(\cdot),\,r}_{A}$, then $\cl$ has the form as in \eqref{e5.20} with a unique $\psi\in\mathcal{BMO}_A^ {p(\cdot),\,q',\,s}$ or $\widetilde{\mathcal{BMO}}_A^{p(\cdot),\,q',\,s}$.
 Moreover,
 $$\|\psi\|_{\widetilde{\mathcal{BMO}}_A^ {p(\cdot),\,q',\,s}}\thicksim\|\psi\|_{\mathcal{BMO}_A^ {p(\cdot),\,q',\,s}}\thicksim
 \|\cl_{\psi}\|_{(\mathcal{H}^{p(\cdot),\,q,\,s,\,r}_{A,\,\rm{atom}})^{*}},$$
 where the implicit positive constants are independent of $\psi$.
\end{theorem}
\begin{remark}

We should point that,
 when
 $p(\cdot):=p\in(0,\,1]$, this result is also new.
\end{remark}
\begin{proof}[Proof of Theorem \ref{t5.1}]
 By Lemmas \ref{l3.1} and \ref{l5.2x}, we only need to show
 $$\mathcal{BMO}_A^ {p(\cdot),\,q',\,s}=(\mathcal{H}^{p(\cdot),\,q,\,s,\,r}_{A,\,\rm{atom}})^{*}.$$
Firstly, we prove that
$$\mathcal{BMO}_A^ {p(\cdot),\,q',\,s}\subset(\mathcal{H}^{p(\cdot),\,q,\,s,\,r}_{A,\,\rm{atom}})^{*}.$$
Let $\psi\in\mathcal{BMO}_A^ {p(\cdot),\,q',\,s}$ and $a$ be a $(p(\cdot),\,q,\,s)$-atom with $\supp a\subset B\in \mathfrak{B}$. Then, by the vanishing moment condition of $a$, H\"{o}lder's inequality and the size condition of $a$, we obtain
\begin{align}\label{e5.j}
\lf|\int_{\rn}\psi(x)a(x)\,dx\r|&=\inf_{P\in P_s}
 \lf|\int_B(\psi(x)-P(x))a(x)\,dx\r|\\\nonumber
 &\leq\|a\|_{L^q}\inf_{P\in P_s}\lf[\int_B|\psi(x)-P(x)|^{q'}\,dx\r]^{1/q'}\\\nonumber
 &\leq\frac{|B|^{1/q}}{\|\chi_B\|_{L^{p(\cdot)}}}\inf_{P\in P_s}\lf[\int_B|\psi(x)-P(x)|^{q'}\,dx\r]^{1/q'}\\\nonumber
 &\leq\|\psi\|_{\mathcal{BMO}_A^ {p(\cdot),\,q',\,s}}.\nonumber
\end{align}
Therefore, for $\{\lambda_i^k\}_{i\in\nn,k\in\zz}\subset\mathbb{C}$ and a sequence $\{a_i^k\}_{i\in\nn,k\in\zz}$ of $(p(\cdot),\,q,\,s)$-atoms supported, respectively, on
$\{x_i+B_{\ell_i^k}\}_{i\in\nn,k\in\zz}\subset\mathfrak{B}$ and $$g=\sum_{k\in\zz}\sum_{i\in\nn}\lambda_i^ka_i^k\in \mathcal{H}^{p(\cdot),\,q,\,s,\,r}_{A,\,\rm{atom}},$$
from  \eqref{e5.j}, we deduce that
\begin{align*}
|\cl_{\psi}(g)|&=\lf|\int_{\rn}\psi(x)g(x)\,dx\r|\leq\sum_{k\in\zz}\sum_{i\in\nn}
\lf|\lambda_i^k\r|\lf|\int_B\lf|\psi(x)-P(x)\r|\lf|a_i^k(x)\r|\,dx\r|\\
&\leq\sum_{k\in\zz}\sum_{i\in\nn}\lf|\lambda_i^k\r|\|\psi\|_{\mathcal{BMO}_A^ {p(\cdot),\,q',\,s}}\\
&\lesssim\lf[\sum_{k\in\zz}2^{kr}\lf\|\sum_{i\in\nn} \chi_{x_i^k+B_{\ell_i^k}}\r\|^r
_{L^{p(\cdot)}}\r]^{1/r}\|\psi\|_{\mathcal{BMO}_A^ {p(\cdot),\,q',\,s}}\\
&\lesssim\|g\|_{ \mathcal{H}^{p(\cdot),\,q,\,s,\,r}_{A,\,\rm{atom}}}
\|\psi\|_{\mathcal{BMO}_A^ {p(\cdot),\,q',\,s}}.
\end{align*}
This implies that $\mathcal{BMO}_A^ {p(\cdot),\,q',\,s}\subset(\mathcal{H}^{p(\cdot),\,q,\,s,\,r}_{A,\,\rm{atom}})^{*}$.

Next we show that $(\mathcal{H}^{p(\cdot),\,q,\,s,\,r}_{A,\,\rm{atom}})^{*}\subset\mathcal{BMO}_A^ {p(\cdot),\,q',\,s}$. For any $B\in\mathfrak{B}$, let
$$S_B: L^1(B)\rightarrow P_s$$
be the natural projection satisfying, for any $g\in L^1$ and $Q\in P_s$,
$$\int_BS_B(g)(x)Q(x)\,dx=\int_Bg(x)Q(x)\,dx.$$
By a similar proof of \cite[(8.9)]{b03}, we obtain that, for any $B\in\mathfrak{B}$ and $g\in L^1(B)$,
$$\sup_{x\in B}\lf|{S}_B(g)(x)\r|\lesssim\frac{\int_B|g(z)|\,dz}{|B|}.$$
Define
$$L^q_0(B):=\lf\{g\in L^q(B): {S}_B(g)(x)=0 \ \mathrm{and}\ g \ \mathrm{is}\ \mathrm{not}\ \mathrm{zero}\ \mathrm{almost} \ \mathrm{everywhere}\r\},$$
 where
$L^q(B):=\{f\in L^q: \mathrm{supp} f\subset B\}$
with $q\in(1,\,\infty]$ and $B\in\mathfrak{B}$,

For any $g\in L^q_0(B)$, set
\begin{eqnarray*}
a(x):=
\frac{|B|^{1/q}}{\|\chi_B\|_{L^{p(\cdot)}}}\|g\|_{L^q(B)}^{-1}g(x)\chi_B(x)                      \end{eqnarray*}
Then $a$ is a $(p(\cdot),\,q,\,s)$-atom. By this and Lemma \ref{l5.2}, we obtain, for any $\cl\in(\mathcal{H}^{p(\cdot),\,q,\,s,\,r}_{A,\,\rm{atom}})^{*}$
and $g\in L^q_0(B)$,
\begin{align}\label{ex3}
|\cl(g)|\leq\frac{\|\chi_B\|_{L^{p(\cdot)}}}{|B|^{1/q}}\|g\|_{L^q(B)}\|\cl\|_
{(\mathcal{H}^{p(\cdot),\,q,\,s,\,r}_{A,\,\rm{atom}})^{*}}.
\end{align}
Thus, by the Hahn-Banach theorem, it can be  extended to a bounded linear functional on $L^q(B)$ with the same norm.

If $q \in(1,\,\infty]$, by the duality of
$L^q(B)$ is $L^{q'}(B)$, we see that there exists a $\Phi\in L^{q'}(B)$ such that, for any $f \in L^q_0(B)$,
$
\cl(f)=\int_Bf(x)\Phi(x)\,dx.
$ In what follows, for any $B\in\mathfrak{B}$, let  $P_s(B)$ denote all the $P_s$ elements vanishing outside $B$.
Now we prove that, if there exists another function $\Phi'\in L^{q'}(B)$ such that, for any
 $f \in L^{q}_0(B)$ and
 $
\cl(f)=\int_Bf(x)\Phi'(x)\,dx,
$
then $\Phi'-\Phi\in P_s(B)$. For this, we only need to show that, if $\Phi,\,\Phi'\in L^1(B)$ such that, for any
$f\in L^\infty_0(B)$, $\int_Bf(x)\Phi'(x)\,dx=\int_Bf(x)\Phi(x)\,dx,$ then $\Phi-\Phi'\in P_s(B)$. In fact, for any
$f \in L^{\infty}_0(B)$, we have
\begin{align*}
0&=\int_B[f(x)-{S}_B(f)(x)][\Phi'(x)-\Phi(x)]\,dx\\
&=\int_Bf(x)[\Phi'(x)-\Phi(x)]\,dx-\int_Bf(x){S}_B(\Phi'(x)-\Phi(x))\,dx\\
&=\int_Bf(x)[\Phi'(x)-\Phi(x)-{S}_B(\Phi'-\Phi)(x)]\,dx.
\end{align*}
Therefore, for a.e. $x\in B\in\mathfrak{B}$, we have
$$\Phi'(x)-\Phi(x)={S}_B(\Phi'-\Phi)(x).$$
Hence $\Phi'-\Phi\in P_s(B)$. From this, we see that, for any $q\in(1,\,\infty]$ and $f\in L^q_0(B)$,
there exists a unique $\Phi\in L^{q'}(B)/P_s(B)$ such that
$
\cl(f)=\int_Bf(x)\Phi(x)\,dx.
$

For any $j\in\nn$ and $g\in L_0^q(B_j)$ with $q\in(1,\,\infty)$, let
$f_j\in L^{q'}(B_j)/P_s(B_j)$ be a unique function such that
$
\cl(g)=\int_{B_j}f_j(x)g(x)\,dx.
$
Then, for any $i,\,j\in\nn$ with $i<j$, $f_j|_{B_i}=f_i$. From this and the fact that, for any
$g\in(\mathcal{H}^{p(\cdot),\,q,\,s,\,r}_{A,\,\rm{atom}})^{*}$, there exists a number $j_0\in\nn$ such that $g\in L_0^q(B_{j_0})$, we conclude that, for any
$g\in(\mathcal{H}^{p(\cdot),\,q,\,s,\,r}_{A,\,\rm{atom}})^{*}$, we have
 \begin{align}\label{ex4}
\cl(g)=\int_B\psi(x)g(x)\,dx,
\end{align}
where $\psi(x):=f_j(x)$ with $x\in B_j$.

Next we show that $\psi\in\mathcal{BMO}_A^ {p(\cdot),\,q',\,s}$. By \cite[(8.12)]{b03}, \eqref{ex3} and \eqref{ex4}, we have that, for any $q\in(1,\,\infty)$,
$B\in\mathfrak{B}$,
 \begin{align*}
\inf_{P\in P_s}\|\psi-P\|_{L^{q'}(B)}=\|\psi\|_{(L_0^q(B))^{*}}\leq\frac{\|\chi_B\|_{L^{p(\cdot)}}}{|B|^{1/q}}
\|\cl\|_{(\mathcal{H}^{p(\cdot),\,q,\,s,\,r}_{A,\,\rm{atom}})^{*}}.
\end{align*}
Therefore, we have that, for any $q\in(1,\,\infty)$,
\begin{align*}
\|\psi\|_{\mathcal{BMO}_A^ {p(\cdot),\,q',\,s}}&=\sup_{B\in\mathfrak{B}}\frac{|B|^{1/q}}
{\|\chi_B\|_{L^{p(\cdot)}}}
\inf_{P\in P_s}\|\psi-P\|_{L^{q'}(B)}
=\sup_{B\in\mathfrak{B}}\frac{|B|^{1/q}}{\|\chi_B\|_{L^{p(\cdot)}}}\|\psi\|_{(L_0^q(B))^{*}}\\
&\leq\|\cl\|_{(\mathcal{H}^{p(\cdot),\,q,\,s,\,r}_{A,\,\rm{atom}})^{*}},
\end{align*}
which implies $\psi\in\mathcal{BMO}_A^ {p(\cdot),\,q',\,s}$.
This finishes the proof of Theorem \ref{t5.1}.
\end{proof}

From Theorem \ref{t5.1}, we easily obtain the following two conclusions.
Moreover, the proof of Corollary \ref{c5.p} is similar to \cite[Lemma 2.21]{zyl16}, we omit the details.

\begin{corollary}\label{c5.p}
 Let $A$ be a given dilation, $p(\cdot)\in C^{\mathrm{log}}$ and $s\in[\lfloor(1/{p_-}-1) {\ln b/\ln \lambda_-}\rfloor,\,\infty)\cap\zz_+$ with $p_-$ as in \eqref{e2.5}. Assume that $f\in\mathcal{BMO}_A^ {p(\cdot),\,1,\,s}$ and $p_+\in(0,\,1]$. Then there exist two positive constants $c_1$ and $c_2$, such that, for any $B\in\mathfrak{B}$ and $\lambda\in(0,\,\fz)$,
 $$\lf|x\in B:|f(x)-P_B^s(f)(x)|>\lambda\r|\leq c_1
 \exp\lf\{\frac{c_2\lambda|B|}
 {\|f\|_{\mathcal{BMO}_A^ {p(\cdot),\,1,\,s}}\|\chi_B\|_{L^{p(\cdot)}}}\r\}.$$
\end{corollary}
\begin{corollary}
 Let $A$ be a given dilation, $p(\cdot)\in C^{\mathrm{log}}$, $p_+\in(0,\,1]$,  $q\in(1,\,\fz)$ and $s\in[\lfloor(1/{p_-}-1) {\ln b/\ln \lambda_-}\rfloor,\,\infty)\cap\zz_+$ with $p_+,\,p_-$ as in \eqref{e2.5}. Then
 $$\mathcal{BMO}_A^ {p(\cdot),\,1,\,s}=\mathcal{BMO}_A^ {p(\cdot),\,q,\,s}$$
 with equivalent quasi-norms.
\end{corollary}

\textbf{Acknowledgements.} The authors would like to express their
deep thanks to the referees for their very careful reading and useful
comments which do improve the presentation of this article.

%\bigskip
%\noindent  Wenhua WANG
%\medskip
%
%\noindent
%School of Mathematics and Statistics\\
%Wuhan University\\
%Wuhan 430072, Hubei, P. R. China\\
%\noindent
%{E-mail }:
%\texttt{wangxjmath@126.com} (Wenhua WANG)\\
%
%\noindent  Aiting WANG (Corresponding author)
%\medskip
%
%\noindent
%School of Mathematics and Statistics\\
%Qinghai Nationalities University\\
%810000, Qinghai, China.\\
%\smallskip
%\noindent{E-mail }:
%\texttt{atwangmath@163.com} (Aiting WANG)\\

\bigskip \medskip

\begin{thebibliography}{30}



\vspace{-0.3cm}
\bibitem{at07}
Abu-Shammala W, Torchinsky A.   The Hardy-Lorentz spaces $H^{p,\,q}(\rn)$. Studia Mathematica  2007; 182: 283-294.

\vspace{-0.3cm}
\bibitem{am02}
Acerbi E, Mingione G. Regularity results for stationary electro-rheological fluids. Archive for Rational Mechanics and Analysis  2002; 164: 213-259.


\vspace{-0.3cm}
\bibitem{b03}
Bownik M.  Anisotropic Hardy spaces and wavelets. Memoirs of the American Mathematical Society  2003; 164: 1-122.

%\vspace{-0.3cm}
%\bibitem{bl76}
%Bergh, J\"{o}ran and L\"{o}fstr\"{o}m, J\"{o}rgen, Interpolation spaces, An introduction. Grundlehren der Mathematischen Wissenschaften, No. 223. %Springer-Verlag, Berlin-New York, (1976). x+207 pp.




\vspace{-0.3cm}
\bibitem{clr06}
Chen Y, Levine S, Rao M.  Variable exponent, linear growth functionals in image restoration. SIAM Journal on Applied Mathematics  2006; 66: 1383-1406.


%\vspace{-0.3cm}
%\bibitem{c64}
%Calder\'{o}n, A.-P., Intermediate spaces and interpolation, the complex method.
%Studia Math. 24 (1964) 113-190.


%\vspace{-0.3cm}
%\bibitem{c75}
%Cwikel, Michael, The dual of weak $L^p$, Ann. Inst. Fourier (Grenoble) 25 (1975), 81-126.

%\vspace{-0.3cm}
%\bibitem{c74}
%R.R. Coifman, {\it A real variable characterization of $H^p$}, Studia Math. {\bf51} (1974), 269-274.

%\vspace{-0.3cm}
%\bibitem{c742}
%R.R. Coifman, {\it Characterization of Fourier transforms of Hardy spaces}, Proc. Natl. Acad. Sci. USA {\bf71} (1974), 4133-4134.

%\vspace{-0.3cm}
%\bibitem{cf81}
%Cwikel, Michael and Fefferman, Charles,
%Maximal seminorms on Weak $L^1$, Studia Math. 69 (1980/81), no. 2, 149-154.

%\vspace{-0.3cm}
%\bibitem{cf84}
%Cwikel, Michael and Fefferman, Charles,
%The canonical seminorm on weak $L^1$, Studia Math. 78 (1984), no. 3, 275-278.




%\vspace{-0.3cm}
%\bibitem{clm93}
%R.R. Coifman, P.-L. Lions, Y. Meyer and S. Semmes,
%Compensated compactness and Hardy spaces,
%J. Math. Pures Appl. (9) 72 (1993), no. 3, 247-286.


\vspace{-0.3cm}
\bibitem{cw77}
 Coifman RR, Weiss G.
 Extensions of Hardy spaces and their use in analysis.
Bulletin of the American Mathematical Society  1977; 83: 569-645.

\vspace{-0.3cm}
\bibitem{cf13}
 Cruz-Uribe DV, Fiorenza A.  Variable Lebesgue spaces, Foundations and harmonic analysis. Applied and Numerical Harmonic Analysis Birkh\"{a}user/Springer Heidelberg 2013; 1-312.

\vspace{-0.3cm}
\bibitem{cw14}
Cruz-Uribe DV, Wang D.  Variable Hardy spaces. Indiana University Mathematics Journal  2014; 63: 447-493.


\vspace{-0.3cm}
\bibitem{dhh11}
Diening L, Harjulehto P, H\"{a}st\"{o} P,  Ru\v{z}i\v{c}ka M.  Lebesgue and Sobolev Spaces with Variable Exponents. Lecture Notes in Mathematics 2011; 1-509.
\vspace{-0.3cm}
\bibitem{f07}
Fan X. Global $C^{1,\,\alpha}$ regularity for variable exponent elliptic equations in divergence form. Journal Differential Equations 2007; 235: 397-417.
\vspace{-0.3cm}
\bibitem{fs86}
Fefferman C, Soria F.  The space weak $H^1$. Studia Mathematica 1986; 85: 1-16.
\vspace{-0.3cm}
\bibitem{fs72}
Fefferman C, Stein EM. $H^p$ spaces of several variables. Acta Mathematica  1972; 129: 137-193.


\vspace{-0.3cm}
\bibitem{jn61}
John F, Nirenberg L.  On functions of bounded mean oscillation. Communications on Pure and Applied Mathematics 1961; 14: 415-426.


\vspace{-0.3cm}
\bibitem{l21}
 Liu X.
 Real-variable characterizations of variable Hardy spaces on Lipschitz domains of $\rn$.
Bulletin of the Korean Mathematical Society 2021; 58: 745-765.
\vspace{-0.3cm}
\bibitem{lx21}
Liu F, Xue Q.
 Weighted estimates for certain rough operators with applications to vector valued inequalities.
Journal Korean Mathematical Society 2021; 58: 1035-1058.


%\vspace{-0.3cm}
%\bibitem{lyy17x}
%J. Liu, F. Weisz, D. Yang and W. Yuan,
%{\it Variable anisotropic Hardy spaces and their appliations},
%Taiwanese J. Math. {\bf22} (2017), 1173-1216.
 \vspace{-0.3cm}
\bibitem{lwyy19}
 Liu J, Weisz F, Yang D,  Yuan W, Littlewood-Paley and finite atomic characterizations of anisotropic variable Hardy-Lorentz spaces and their applications. Journal of Fourier Analysis and Applications 2019; 25: 874-922.

\vspace{-0.3cm}
\bibitem{lyy16}
Liu J, Yang D, Yuan W.  Anisotropic Hardy-Lorentz spaces and their applications. Science China Mathematics 2016; 59: 1669-1720.

\vspace{-0.3cm}
\bibitem{lyy17}
 Liu J, Yang D, Yuan W.
 Anisotropic variable Hardy-Lorentz spaces and their real interpolation.
Journal of Mathematical Analysis and Applications  2017; 456: 356-393.



%\vspace{-0.3cm}
%\bibitem{mtt03}
%Muscalu, Camil; Tao, Terence and Thiele, Christoph,
%A counterexample to a multilinear endpoint question of Christ and Kiselev,
%Math. Res. Lett. 10 (2003), no. 2-3, 237-246.


%\vspace{-0.3cm}
%\bibitem{mr15}
%Merker, Jochen and Rakotoson, Jean-Michel,
%Very weak solutions of Poisson's equation with singular data under Neumann boundary conditions,
%Calc. Var. Partial Differential Equations 52 (2015), no. 3-4, 705-726.


\vspace{-0.3cm}
\bibitem{ns12}
Nakai E, Sawano Y. Hardy spaces with variable exponents and generalized Campanato spaces. Journal of Functional Analysis 2012; 262: 3665-3748.

%\vspace{-0.3cm}
%\bibitem{o31}
%W. Orlicz, {\it $\mathrm{\ddot{U}}$ber konjugierte Exponentenfolgen}, Studia Math. {\bf3} (1931), no. 1, 200-211.
%\vspace{-0.3cm}
%\bibitem{p05}
%Parilov, D. V.,
%Two theorems on the Hardy-Lorentz classes $H^{1,\,q}$,
%Zap. Nauchn. Sem. S.-Peterburg. Otdel. Mat. Inst. Steklov. (POMI) 327 (2005), Issled. po Line\v{i}n. Oper. i Teor. Funkts. 33, 150-167, 238; translation in
%J. Math. Sci. (N.Y.) 139 (2006), no. 2, 6447-6456.


%\vspace{-0.3cm}
%\bibitem{p15}
%Phuc, Nguyen Cong, The Navier-Stokes equations in nonendpoint borderline Lorentz spaces, J. Math. Fluid Mech. 17 (2015), no. 4, 741-760.

\vspace{-0.3cm}
\bibitem{s13}
Sawano Y. Atomic decompositions of Hardy space with variable exponent and its application to bounded linear operators. Integral Equations Operator Theory  2013; 77: 123-148.


\vspace{-0.3cm}
\bibitem{s93}
Stein EM.  Harmonic Analysis: Real-Variable
Methods, Orthogonality, and Oscillatory Integrals. Princeton Univ Press Princeton N J 1993.

\vspace{-0.3cm}
\bibitem{s60}
Stein EM, Weiss G. On the theory of harmonic functions of several variables.  The theory of $H^p$-spaces. Acta Mathematica 1960; 103: 25-62.


\vspace{-0.3cm}
\bibitem{sa89}
 Str\"{o}mberg JO, Torchinsky A.  Weighted Hardy spaces. Lecture Notes in Mathematics Springer-Verlag Berlin 1989.



\vspace{-0.3cm}
\bibitem{tw80}
Taibleson MH, Weiss G,  The molecular characterization of certain Hardy spaces,  Representation theorems for Hardy
spaces, Ast¨¦risque, Society Mathematical France, Paris, 1980; 77: 67-149.
\vspace{-0.3cm}
\bibitem{t19}
 Tan J. Atomic decompositions of localized Hardy spaces with
variable exponents and applications. Journal of Geometric Analysis  2019; 29: 799-827.

\vspace{-0.3cm}
\bibitem{t17}
Tang L.  $L^{p(\cdot),\,\lambda(\cdot)}$ regularity for fully nonlinear elliptic equations. Nonlinear Analysis  2017; 149: 117-129.
%\vspace{-0.3cm}
%\bibitem{tw80}
%M.H. Taibleson and G. Weiss, The Molecular Characterization of Certain Hardy Spaces, Representation Theorems for Hardy Spaces, Ast¨¦risque, vol. 77, pp. 67-149. Soci¨¦t¨¦ math¨¦matique, Paris
%(1980)
\vspace{-0.3cm}
\bibitem{w21}
 Wang W. Dualities of variable anisotropic Hardy spaces and boundedness of singular integral operators. Bulletin of the Korean Mathematical Society 2021; 58: 365-384.

\vspace{-0.3cm}
\bibitem{wl12}
Wang H, Liu Z.  The Herz-type Hardy spaces with variable exponent and their
applications.  Taiwanese Journal of Mathematics  2012; 16: 1363-1389.

\vspace{-0.3cm}
\bibitem{zyl16}
Zhuo C, Yang D, Liang Y.  Intrinsic square function characterization of Hardy spaces with variable exponents. Bulletin of the Malaysian Mathematical Sciences Society 2016; 39: 1541-1577.

\end{thebibliography}
\end{document}